\documentclass[11pt, 
]{amsart}

\usepackage[a4paper,margin=4cm]{geometry}
\usepackage{amsfonts,amssymb,amscd}
\usepackage{graphicx,mathrsfs} 
\usepackage{subfigure}
\newtheorem{theorem}{Theorem}[section]
\newtheorem{lemma}[theorem]{Lemma}
\newtheorem{corollary}[theorem]{Corollary}

\newtheorem{definition}[theorem]{Definition}

\theoremstyle{remark}

\newtheorem{remark}[theorem]{\bf Remark}
\newtheorem{example}[theorem]{\bf Example}

\renewcommand{\leq}{\leqslant}
\renewcommand{\geq}{\geqslant}

\newcommand{\ptl}{\partial}

\newcommand{\rr}{\mathbb{R}}

\newcommand{\nn}{\mathbb{N}}
\newcommand{\C}{\mathscr{C}_2}


\numberwithin{equation}{section}

\begin{document}

\title{Borsuk number for planar convex bodies}

\author[A. Ca\~nete]{Antonio Ca\~nete}
\address{Departamento de Matem\'atica Aplicada I \\ Universidad de Sevilla}
\email{antonioc@us.es}

\author[U. Schnell]{Uwe Schnell}
\address{Faculty of Mathemics and Sciences \\ University of Applied Sciences Zittau/G{\"or}litz}
\email{uschnell@hszg.de}

\thanks{The first author is partially supported by the MICINN project MTM2017-84851-C2-1-P,
and by Junta de Andaluc\'ia grant FQM-325.}


\maketitle

\begin{abstract}
By using some simple tools from graph theory,
we obtain a characterization of the planar convex bodies with Borsuk number equal to two.
This result allows to give some examples of planar convex bodies with Borsuk number equal to three.
Moreover, we also prove that the unique centrally symmetric planar convex bodies
with Borsuk number equal to three are the Euclidean balls.



\end{abstract}

\section{Introduction}
In 1933, K.~Borsuk formulated the following question \cite{borsuk}:
\begin{quotation}
\emph{Let $C$ be a bounded set in $\rr^n$.
Is it possible to divide $C$ into $n+1$ subsets with diameters
strictly smaller than the diameter of $C$?}
\end{quotation}
This question has become a classical problem in Geometry~\cite[\S.~D14]{cfg},
being deeply studied during the last century.
In the planar case ($n=2$), Borsuk gave an affirmative answer in his original paper~\cite{borsuk},
basing the proof in the following nice result by P\'al~\cite{pal}, see also~\cite[Lemma~1.1]{bolt-gohb}:
any planar set $C$ with diameter $h>0$ is contained in a regular hexagon $H$ of width equal to $h$.
It is not difficult to check that $H$ can be divided into three congruent subsets with diameters less than $h$,
and consequently, the induced division for $C$ will satisfy the same property.
For $n=3$, the answer is also affirmative and was proven by several authors with different techniques~\cite{perkal,egg,g,heppes}.
Moreover, the same property holds in $\rr^n$ when $C$ is compact, convex with smooth boundary~\cite{had},
or when $C$ is compact, convex and centrally symmetric~\cite{riesling}.
However, Borsuk's question is not true in general, as the celebrated counterexample by Kahn and Kalai in 1993 exhibits \cite{kk}.
In that work, they consider an equivalent discrete formulation of the problem (see~\cite{larman}),
and they deduced, by using a combinatorial result by Frankl and Wilson~\cite{fw},
that Borsuk's question is not true in $\rr^n$ for $n\geq 2015$
(indicating also that the same happens in $\rr^{1325}$).
This was the initial point for a \emph{sort of competition} searching for the least-dimension Euclidean space
for which Borsuk's question does not hold \cite{nilli,raigo,g-w,hinrichs,pikhurko,h-r}.
Up to our knowledge, the last recent results in this direction show that the answer is negative
in $\rr^{65}$~\cite{bondarenko} and in $\rr^{64}$~\cite{j-b}.
In these last two works, the authors give elaborated counterexamples
by considering a particular strongly regular graph and constructing a finite set
of points in the corresponding sphere (in fact, the second reference is
a refinement of the first one).

Particularizing in $\rr^2$, the affirmative answer to Borsuk's question
implies that any planar bounded set $C$ can be divided into
three (or less) subsets with diameters strictly smaller than the diameter of $C$.
It is clear that there are planar sets which can be
partitioned into \emph{two} subsets with smaller diameters (for instance, consider a square, whose diameter is attained by the distance between any two
non-consecutive vertices, and the division determined by any horizontal line, see Figure~\ref{fig:2}).
\begin{figure}[ht]
\centering
  \includegraphics[width=0.5\textwidth]{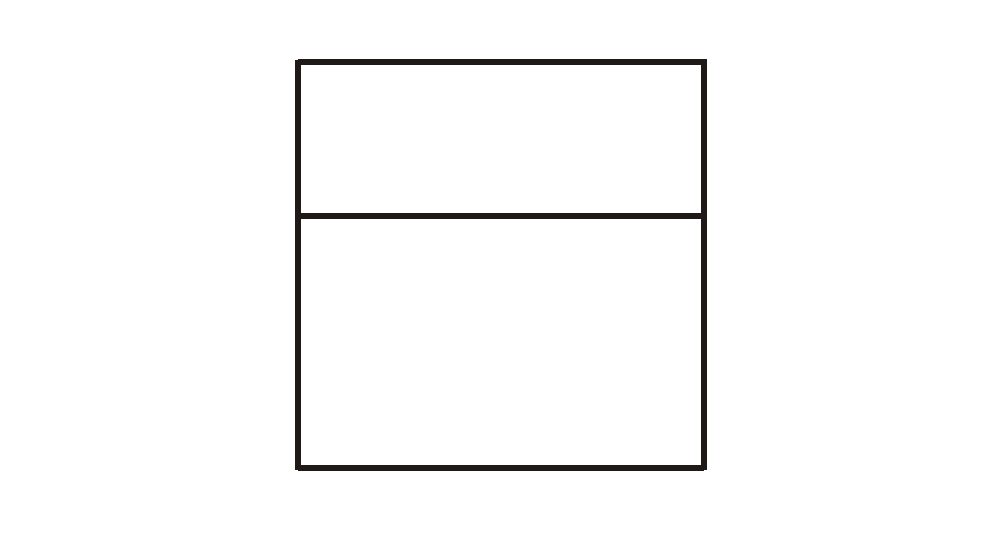}\\
  \caption{A division of a square into two subsets with strictly smaller diameters}\label{fig:2}
\end{figure}
This fact leads to the following reformulation of the problem, which will be stated for general dimension:
for a given bounded set $C$ in $\rr^n$, let $\alpha(C)$ denote the minimal number of subsets 
with strictly smaller diameters
into which $C$ can be decomposed. This number is usually called \emph{Borsuk number} of $C$.
The original question by Borsuk can be thus rewritten as:
\begin{quotation}
\begin{center}
\emph{Let $C$ be a bounded set in $\rr^n$. Is it true that $\alpha(C)\leq n+1$?}
\end{center}
\end{quotation}

The calculation of the precise value of Borsuk number for a given bounded set
also constitutes an interesting issue
and has been considered in different texts in literature, see for instance~\cite{w,bolt-gohb}.
Moreover, in~\cite[\S.~4]{kolo} we can find a nice equivalent formulation of Borsuk's question
in terms of the Minkowski sum of two convex compact sets of constant width,
based on the precise values of the corresponding Borsuk numbers. 

As noted previously, we know that $\alpha(C)$ is less than or equal to three
for any planar bounded set $C$ (since Borsuk's question is true).
As described in~\cite[\S.~1.4]{bolt-gohb},
it is natural investigating when the equality to three occurs.
The first characterization of the planar sets with Borsuk number equal to three is due to Boltyanskii~\cite{bolt},
by means of the notion of \emph{completion of constant width}:
P\'al proved that any set $C$ (in arbitrary dimension) with diameter $h>0$ can be covered by a compact set with
constant width equal to $h$~\cite{pal}, see also~\cite[\S.~15.64]{b-f}. This constant-width set is called a
completion of $C$, and is not unique in general.
Boltyanskii's characterization (whose proof is purely geometrical) states that $\alpha(C)=3$ if and only if the completion of $C$ is unique~\cite[Th.~1.3]{bolt-gohb}.
Unfortunately, it is not easy to check in practice whether the completion of a general set is unique,
and so the applicability of this result is not very extensive.
Later on, K.~Ko{\l}odziejczyk provided a new characterization in the convex setting:
for a given planar convex compact set $C$,
a \emph{diameter segment} of $C$ will be a segment contained in $C$ with length equal to the diameter of $C$.
Then, $\alpha(C)=2$ if and only if there exists a non-diameter segment in $C$
which intersects the interior of any diameter segment of $C$ \cite[Th.~3.1]{kolo}.
A definitive characterization holds when restricting
to the family $\mathscr{C}_2$ of \emph{centrally symmetric} planar convex compact sets:
D.~Ko{\l}odziejczyk proved that, if $C\in\C$, then $\alpha(C)=3$ if and only if $C$ is an Euclidean ball~\cite[Appendix]{danuta}.
We note that this nice result is stated for a wider family of sets,
namely, the planar convex bodies with all the diameter segments being concurrent at one point,
and additionally, for general dimension.


%

In these notes we introduce an alternative approach to these questions,
by using some tools from graph theory.
This point of view has not been used when treating this kind of problems,
and it might provide some advances in future.
We have organized these notes into two separate sections.
Section~\ref{se:centrally} will be devoted to the family $\C$ of centrally symmetric planar convex compact sets,
and in Section~\ref{se:non} we will consider general planar convex compact sets.
We have chosen this structure because we think it will show
the remarkable differences appearing due to the lack of symmetry.

In Section~\ref{se:centrally} we will recover the previous detailed results in $\C$
by means of the \emph{diameter graph} of a set, as follows.
For a given planar compact set $C$, we can consider the diameter graph $G_C=(V,E)$ associated to $C$,
whose set of vertices $V$ is composed by the endpoints of the diameter segments of $C$,
and whose set of edges $E$ is composed precisely by such diameter segments, see~\cite{dol,fprr}.
With this notation, our Lemma~\ref{le:union} provides a characterization
for the sets in $\C$ with Borsuk number equal to two,
expressed in terms of the associated diameter graph:
for any $C\in\C$, we have that $\alpha(C)=2$ if and only if $V\neq\ptl C$.
This result leads to the known fact that the only centrally symmetric planar convex bodies
with Borsuk number equal to three are the Euclidean balls (see Theorem~\ref{th:main}).

Finally, in Section~\ref{se:non} we focus on general planar compact convex sets.
Some similar reasonings involving the diameter graph will allow us to obtain Theorem~\ref{prop:non},
where a characterization for the sets with Borsuk number equal to two is established.
In particular, this result easily provides several examples of planar sets with Borsuk number equal to three,
different from an Euclidean ball (some of them are depicted in Figures~\ref{fig:3} and \ref{fig:pentagon2}).

\section{Some basic definitions}

In this section we give some simple definitions for planar compact sets which will be used along this paper.

Let $C\subset\rr^2$ be a compact set (also referred to as a body, as usual), and denote by $d$ the Euclidean distance in $\rr^2$. Let
$$D(C)=\max\{d(x,y):x,y\in C\}$$
be the diameter of $C$, that is, the maximal distance between two points in $C$.
Any line segment $\overline{x\,y}$ with endpoints $x,y\in\ptl C$ satisfying that $d(x,y)=D(C)$
will be called \emph{a diameter segment} of $C$.
This notion leads us to a particular planar graph associated to any planar compact set in the following way, see \cite{bowers,dol,fprr}.

\begin{definition}
\label{def:dg}
Let $C$ be a planar compact set.
The diameter graph $G_C=(V,E)$ associated to $C$ is the planar graph
whose vertices are the points in $\ptl C$ which are endpoints of the diameter segments of $C$, and whose edges are the diameter segments. 
In other words,
\begin{align*}
V=\{x_i&\in\ptl C: \exists\ y_i\in\ptl C \text{ with } d(x_i,y_i)=D(C)\}, \text{ and }
\\[2mm]
&\overline{x_i\,x_j} \in E \text{ if and only if } d(x_i,x_j)=D(C).
\end{align*}
\end{definition}

\begin{definition}
\label{def:Bn}
Let $C$ be a planar compact set.
The Borsuk number of $C$ is defined as the minimal number $\alpha(C)\in\nn$ satisfying that $C$ can be divided
into $\alpha(C)$ subsets, all of them with diameters strictly smaller than $D(C)$.
\end{definition}

\begin{remark}
We recall that, since Borsuk's question is true in $\rr^2$, then $\alpha(C)$ is at most three for any planar compact set $C$.
\end{remark}

\begin{remark}
We point out that for an \emph{unbounded} planar set $C$,
the diameter $D(C)$ could be defined by using the \emph{supremum} instead of the maximum above,
thus yielding $D(C)=\infty$.
In this situation, any division of $C$ into $k$ subsets
($k\in\nn$, $k\geq 2$) will have at least one unbounded subset, with infinite diameter.
Consequently we cannot find a division of $C$ with strictly smaller diameters in this case.
\end{remark}

\section{Centrally symmetric planar convex bodies}
\label{se:centrally}

Throughout this section we will focus on the family $\C$ of centrally symmetric planar convex bodies.
We will see that the symmetry assumption allows us to completely characterize the sets of this family
whose Borsuk number is equal to three.

Consider a planar convex body $C$ which is centrally symmetric.
Recall that this means that $C$ is invariant under the action of the rotation of angle $\pi$
centered at a point $p\in C$, which will be called the \emph{center of symmetry} of $C$.
We first prove some properties for the diameter graph associated to $C$ in this setting.

\begin{lemma}
\label{le:through}
Let $C\in\C$, with $p$ its center of symmetry.
Then, any diameter segment of $C$ passes through $p$.
\end{lemma}

\begin{proof}
Let $\overline{x\,y}$ be a diameter segment of $C$, and assume that it does not pass through $p$.
Let $\widetilde{p}$ be the intersection point between $\overline{x\,y}$
and the line perpendicular to $\overline{x\,y}$ passing through $p$, see Figure~\ref{fig:line}.
\begin{figure}[h]
\centering
	\includegraphics[width=0.55\textwidth]{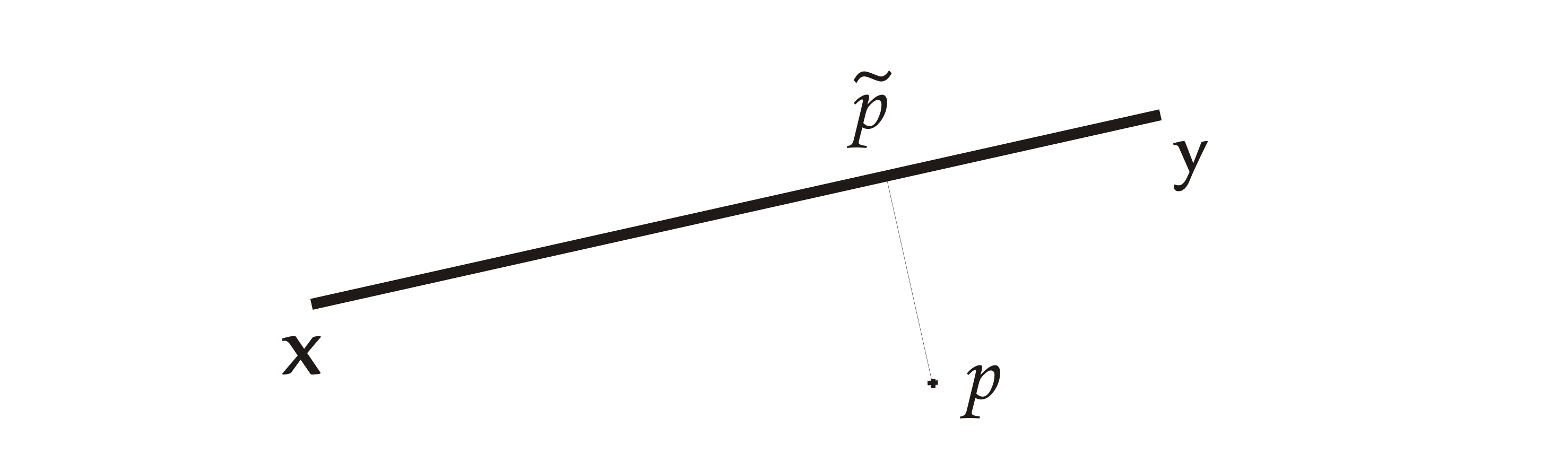}
	\caption{A diameter segment not containing $p$}
    \label{fig:line}
\end{figure}
Without loss of generality, we can assume that $d(x,\widetilde{p})\geq d(y,\widetilde{p})$.
Denote by $x'\in\ptl C$ the symmetric point of $x$ with respect to $p$. Then,
\begin{align*}
d(x,x')&=2\,d(x,p)>2\,d(x,\widetilde{p})\geq d(x,\widetilde{p})+d(y,\widetilde{p})=d(x,y)=D(C),
\end{align*}
which is a contradiction. Thus $\overline{x\,y}$ necessarily passes through $p$, as stated.
\end{proof}

\begin{lemma}
\label{le:bipartite}
Let $C\in\C$. Then, the diameter graph $G_C$ associated to $C$ is bipartite.
\end{lemma}

\begin{proof}
From Lemma~\ref{le:through}, any diameter segment of $C$ passes through the center of symmetry of $C$,
which necessarily implies that each vertex of $G_C$ has a unique incident edge, and so it is a vertex of degree one.
This directly gives that $G_C$ is bipartite.
\end{proof}

\begin{remark}
\label{obs:altern}
Consider a centrally symmetric planar convex body $C$, with $p$ its center of symmetry,
and its associated diameter graph $G_C=(V,E)$.
Lemma~\ref{le:bipartite} asserts that $G_C$ is bipartite, which means that $V$ can be decomposed
into two disjoint subsets $V_R$, $V_B \subset\ptl C$,
satisfying that the endpoints of each edge of $G_C$ lie in different subsets.
Moreover, $V_R$ and $V_B$ are symmetric with respect to $p$,
and without loss of generality, we can assume that the vertices in $V_R$
are all placed together along $\ptl C$, with no alternation with the vertices in $V_B$.
\end{remark}

\begin{remark}
The reader may compare Lemma~\ref{le:bipartite} above with \cite[Th.~1]{dol}, where it is shown, in the general case
(i.e., without the assumption of central symmetry),
that the diameter graph minus one (particular) vertex is bipartite.
\end{remark}

The following result characterizes the centrally symmetric planar convex bodies with Borsuk number equal to two,
by means of the corresponding diameter graph. 

\begin{lemma}
\label{le:union}
Let $C\in\C$, and let $G_C=(V,E)$ be the diameter graph associated to $C$.
Then,
$\alpha(C)=2$ if and only if $V\neq\ptl C$.
\end{lemma}

\begin{proof}
Assume firstly that $\alpha(C)=2$,
and let $\{C_1,\,C_2\}$ be a division of $C$ with $D(C_i)<D(C)$, $i=1,2$.
We can also assume that the division is determined by a curve in $C$
with endpoints $v_1$, $v_2\in\ptl C$.
It follows that $v_1\notin V$
(otherwise, its symmetric point $v_1'$ will also belong to  $V$, by Lemma~\ref{le:through},
and so $D(C_j)\geq d(v_1,v_1')=D(C)$, where $C_j$ is the subset containing $v_1'$),
and therefore $V \neq \ptl C$.

Assume now that $V\neq\ptl C$, and consider $x\in\ptl C$ such that $x\notin V$.
By using again Lemma~\ref{le:through}, we have that its symmetric point $x'$ will have the same property.
The division of $C$ given by the line segment $\overline{x\,x'}$ satisfies that
the two corresponding subsets have diameter strictly less than $D(C)$,
since no edge of $G_C$ will be contained in any of them
(in fact, $\overline{x\,x'}$ splits any diameter segment of $C$),
and so $\alpha(C)=2$.
\end{proof}

\begin{remark}
If a centrally symmetric planar convex body $C$ has associated diameter graph with finite set of vertices,
then Lemma~\ref{le:union} trivially holds, and so $C$ will have Borsuk number equal to two.
\end{remark}

We can now state the following result,
which proves that the unique centrally symmetric planar convex bodies
with Borsuk number equal to three are the Euclidean balls.

\begin{theorem}
\label{th:main}
Let $C\in\C$. Then, $\alpha(C)=3$ if and only if $C$ is an Euclidean ball.
\end{theorem}

\begin{proof}
It is well known that any planar Euclidean ball has Borsuk number equal to three \cite{borsuk}.
On the other hand, consider the diameter graph $G_C=(V,E)$ associated to $C$,
with $V=\ptl C$ in view of Lemma~\ref{le:union}. %
Hence, for any $z\in\ptl C$, we have that $z\in V$ and so $d(z,z')=D(C)$,
where $z'\in\ptl C$ is the symmetric point of $z$ with respect to the center of symmetry $p$ of $C$.
This implies that $d(z,p)=1/2\,D(C)$, and so $z\in\ptl B(p,1/2\,D(C))$,
where $B(q,r)$ denotes the Euclidean ball centered at $q$ of radius $r$.
This holds for all $z\in\ptl C$, and consequently, $\ptl C=\ptl B(p,1/2\,D(C))$, as desired.
\end{proof}

\begin{remark}
As noted in the Introduction, some characterizations in the spirit of Lemma~\ref{le:union}
and Theorem~\ref{th:main}
have appeared previously in literature. Boltyanskii~\cite[Th.~1.3]{bolt-gohb} proved that
a planar set has Borsuk number equal to three if and only if
the corresponding completion to a constant width set is unique.
Unfortunately, this nice result seems hard to be applied
(in general, given a planar set, it is difficult
to check the uniqueness of a completion of this type).
Later on, a result by K.~Ko{\l}odziejczyk \cite[Th.~3.1]{kolo}
shows that a planar convex body $C$ has Borsuk number equal to two
if and only if there exists a chord in $C$ intersecting any diameter segment of $C$,
which is equivalent to our Lemma~\ref{le:union} in the centrally symmetric case.
The proof of this result is based on set theory and,
in particular, on Zorn's Lemma~\cite[Appendix~2, Cor.~2.5]{lang}.
More generally, D.~Ko{\l}odziejczyk~\cite[Appendix]{danuta} proved that $\alpha(C)=n+1$ if and only if $C$ is an Euclidean ball,
for any convex body $C\subset\rr^n$ whose diameter segments have a common point,
which includes our Theorem~\ref{th:main}.
\end{remark}

\begin{remark}
A relative optimization problem involving the diameter functional, treated in~\cite{bisections} (see also~\cite{mejora}),
has a direct connection with Theorem~\ref{th:main}.
Those papers focus on centrally symmetric planar convex bodies,
searching for the divisions into two subsets that \emph{minimize} the \emph{maximum relative diameter}
(which is defined as the maximum of the diameters of the two subsets of the division).
Our Theorem~\ref{th:main} yields that the unique sets
which are not suitable for this problem are Euclidean balls:
any division ${C_1,C_2}$ of a given ball $\mathcal D$
will satisfy that $D(C_1)=D(C_2)=D(\mathcal D)$, since $\alpha(\mathcal D)=3$.
Therefore, the maximum relative diameter functional is constant for all the divisions of $\mathcal{D}$,
and this minimization problem is not meaningful for this particular set
(any division of $\mathcal D$ can be considered minimizing).
\end{remark}

\section{General planar convex bodies}
\label{se:non}

In this section we will focus on planar convex bodies, removing the central symmetry hypothesis.
This more general setting presents some remarkable differences. In particular, the associated diameter graph
is not always bipartite: this can be easily seen by considering, for instance, an equilateral triangle
(or some isosceles ones).
Theorem~\ref{prop:non} below gives a characterization
of the planar convex bodies with Borsuk number equal to two,
and allows to find examples with Borsuk number equal to three,
which are different from the Euclidean balls.

\begin{theorem}
\label{prop:non}
Let $C$ be a planar convex body, and let $G_C=(V,E)$ be the diameter graph associated to $C$.
Then, $\alpha(C)=2$ if and only if $G_C$ is bipartite, with a decomposition $V=V_R\cup V_B$
(as indicated in Remark~\ref{obs:altern})
such that the closures of $V_R$ and $V_B$ have empty intersection (that is, $\overline{V_R}\cap\overline{V_B}=\emptyset$).
\end{theorem}

\begin{proof}
Assume firstly that $\alpha(C)=2$, and let $\{C_1,C_2\}$ be a division of $C$ with $D(C_i)<D(C)$, $i=1,2$,
determined by a curve with endpoints $v_1$, $v_2\in\ptl C$.
Then, by considering $V_R=V\cap\ptl C_1$ and $V_B=V\cap\ptl C_2$,
we will have a decomposition of $V$ which implies that $G_C$ is a bipartite graph:
if an arbitrary edge $\overline{x\,y}$ in $G_C$ has both vertices $x$, $y$ in (say) $V_R$, then
$$D(C_1)\geq d(x,y)=D(C),$$
yielding a contradiction.
In addition, the previous decomposition gives that $\overline{V_R}\cap \overline{V_B}\subseteq\{v_1,\,v_2\}$.
Suppose that $v_1\in\overline{V_R}\cap \overline{V_B}$.
In particular, there exists a sequence $\{x_R^n\}_n\subset V_R$ which converges to $v_1$.
Moreover, for each element $x_R^n$ of the sequence, there exists $x_B^n\in V_B$ with
\begin{equation}
\label{eq:distance}
d(x_R^n,x_B^n)=D(C).
\end{equation}
Note that the sequence $\{x_B^n\}_n$ is contained in $\ptl C$, which is compact,
so we can assume that $\{x_B^n\}_n$ converges to a certain $w\in\ptl C$.
Without loss of generality, we can assume that $w\in C_1$.
Taking now limit in \eqref{eq:distance} when $n$ tends to infinite,
it follows that $D(C_1)\geq d(v_1,w)=D(C)$,
which is again a contradiction. So necessarily $\overline{V_R}\cap \overline{V_B}=\emptyset$, as stated.

Now assume that $G_C$ is bipartite with $V=V_R\cup V_B$ as in Remark~\ref{obs:altern},
satisfying $\overline{V_R}\cap\overline{V_B}=\emptyset$.
These hypotheses imply that we can find two compact connected curves $\beta_1$, $\beta_2$ contained in $\ptl C$
such that $V_R$, $\beta_1$, $V_B$, $\beta_2$ appear successively along $\ptl C$ in this order.
Take $x_i\in\beta_i$, for $i=1,2$. The line segment $\overline{x_1\,x_2}$ divides $C$
into two subsets with diameters smaller than $D(C)$, and so $\alpha(C)=2$.
\end{proof}

\begin{remark}
We point out that, when a diameter graph $(V,E)$ is bipartite,
we cannot directly deduce that there is a decomposition $V=V_R\cup V_B$
with no alternation between the vertices of $V_R$ and the vertices of $V_B$
(as indicated in Remark~\ref{obs:altern}).
\end{remark}

\begin{remark}
Any planar convex body which does not satisfy the conditions stated in Theorem~\ref{prop:non}
will have Borsuk number equal to three.
This is the case for any regular polygon with an odd number of vertices, or any Reuleaux polygon
(it is easy to check that the corresponding diameter graphs are not bipartite).
Moreover, any small convex variation of the previous sets preserving the corresponding diameters segments
will have the same property.

\begin{figure}[h]
\centering
	\includegraphics[width=0.62\textwidth]{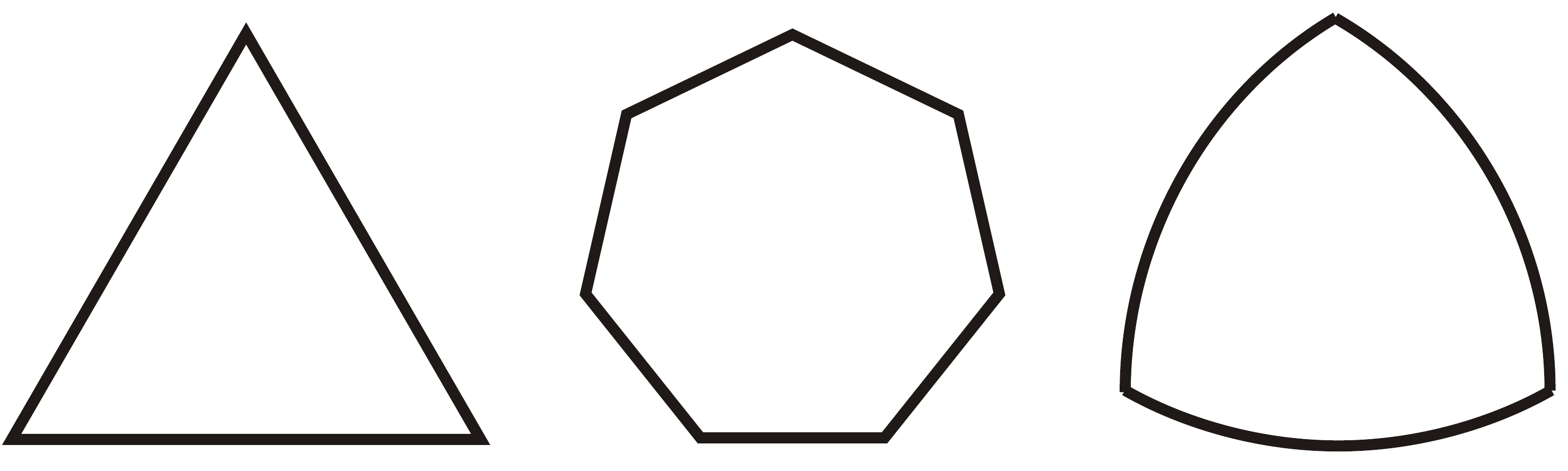}
	\caption{Some planar convex bodies with Borsuk number equal to three}
\label{fig:3}
\end{figure}

\end{remark}

The following Example~\ref{ex:pentagons} describes the construction of
a family of planar convex bodies with Borsuk number equal to three.
\begin{example}
\label{ex:pentagons}
Let $C_1$ be a circle centered at the origin $o=(0,0)$ with radius $r>0$.
Let $a=(r,0)$, and consider $C_2$ the circle centered at $a$ with radius $r$.
Call $L$ the symmetric lens given by the intersection of the circles $C_1$ and $C_2$.
Choose two arbitrary points $b\in L\cap C_2$, and $c\in L\cap C_1$,
both of them in the upper half-plane,
and call $d$ be the intersection point of the two circles centered at $b$ and $c$ with radius $r$,
which is contained in $L$.
Then, the convex hull $C$ of points $\{o,a,b,c,d\}$ has five diameters segments,
namely $\overline{oa}$, $\overline{ab}$, $\overline{bd}$, $\overline{cd}$ and $\overline{oc}$,
and the associated diameter graph $G_C$ is not bipartite.
By applying Theorem~\ref{prop:non}, we conclude that $\alpha(C)=3$.
Figure~\ref{fig:pentagon2} below shows a particular set of this family for radius $r=4$.
We also note that slight convex variations of $C$ will provide new non-polynomial examples with
Borsuk number equal to three.
\end{example}

\begin{figure}[ht]
\centering
	\includegraphics[width=1.2\textwidth]{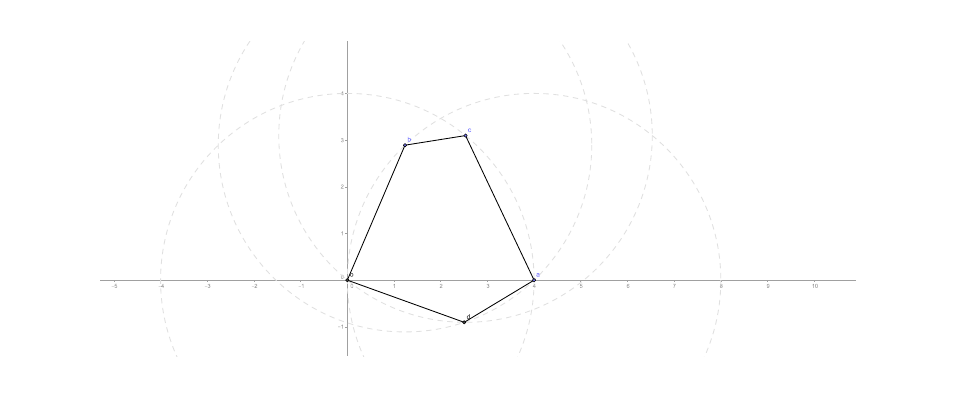}
	\caption{Another planar convex body with Borsuk number equal to three}
\label{fig:pentagon2}
\end{figure}

Although Borsuk's original question is not stated in a convex setting,
the following related problem can be posed:
given a \emph{convex} body $C$ in $\rr^n$,
is it possible to divide $C$ into $n+1$ \emph{convex} subsets with strictly smaller diameters than $C$?
Corollary~\ref{co:convex} below gives an affirmative answer to this question in the planar case.

\begin{corollary}
\label{co:convex}
Let $C$ be a planar convex body, with Borsuk number $\alpha(C)$.
Then $C$ can be divided into $\alpha(C)$ convex subsets with diameters strictly smaller than $D(C)$.
\end{corollary}

\begin{proof}
Recall that $\alpha(C)\in\{2,3\}$ in the planar case.
If $\alpha(C)=3$, the original reasoning by Borsuk (already outlined in the Introduction) proves the claim:
$C$ is contained in the regular hexagon $H$ of width equal to $D(C)$,
and $H$ can be divided into three appropriate congruent subsets by using three line segments.
Then, the induced division of $C$ will consist of three convex subsets with smaller diameters,
due to the convexity of $C$. 
Assume now that $\alpha(C)=2$.
Then, there exists a curve $\gamma$ in $C$
which provides a division of $C$ into two subsets with smaller diameters.
It is clear that the line segment joining the endpoints of $\gamma$
also provides a division of $C$ with the same property.
The two subsets determined by this division are convex, again because $C$ is convex, as desired.
\end{proof}

%


%



\end{document}